\documentclass{article}
\usepackage{graphicx} 
\usepackage{amsthm}
\usepackage{amssymb}
\usepackage{amsmath}
\usepackage{xcolor}
\usepackage{enumerate}
\usepackage{microtype}
\usepackage[numbers]{natbib}
\usepackage{hyperref}

\newtheorem{theorem}{Theorem}

\newtheorem{lemma}[theorem]{Lemma}

\newtheorem{corollary}[theorem]{Corollary}

\newtheorem{conjecture}{Conjecture}

\def\inst#1{$^{#1}$}

\title{On forbidden configurations in \\point-line incidence graphs\thanks{M. Balko was supported by grant no. 23-04949X of the Czech Science Foundation (GA\v{C}R) and by the Center for Foundations of Modern Computer Science (Charles Univ. project UNCE 24/SCI/008). N. Frankl was partially supported by ERC Advanced Grant "GeoScape".
} 
}
\author{
Martin Balko\inst{1}
\and
N\'{o}ra Frankl\inst{2}}

\begin{document}

\maketitle

\begin{center}
{\footnotesize
\inst{1} 
Department of Applied Mathematics, \\
Faculty of Mathematics and Physics, Charles University, Czech Republic \\
\texttt{balko@kam.mff.cuni.cz}\\
\inst{2} 
School of Mathematics and Statistics, The Open University, UK, \\and Alfr\'{e}d R\'{e}nyi Institute of Mathematics, Hungary \\
\texttt{nora.frankl@open.ac.uk}\\
}
\end{center}

\begin{abstract}
The celebrated \emph{Szemer\'edi--Trotter theorem} states that the maximum number of incidences between $n$ points and $n$ lines in the plane is
$O(n^{4/3})$, which is asymptotically tight.
Solymosi (2005) conjectured that for any set of points $P_0$ and for any set of lines $\mathcal{L}_0$ in the plane, the maximum number of incidences between $n$ points and $n$ lines in the plane whose incidence graph does not contain the incidence graph of $(P_0,\mathcal{L}_0)$ is $o(n^{4/3})$.
This conjecture is mentioned in the book of Brass, Moser, and Pach (2005). Even a stronger conjecture, which states that the bound can be improved to $O(n^{4/3-\varepsilon})$ for some $\varepsilon = \varepsilon(P_0,\mathcal{L}_0)>0$, was introduced by Mirzaei and Suk (2021).
We disprove both of these conjectures.
We also introduce a new approach for proving the upper bound $O(n^{4/3-\varepsilon})$ on the number of incidences for configurations $(P,\mathcal{L})$ that avoid certain subconfigurations.
\end{abstract}

\section{Introduction}

Let $P$ be a set of $m$ points and $\mathcal{L}$ be a set of $n$ lines in the plane.
We define $I(P,\mathcal{L})$ to be the set of \emph{incidences} between points from $P$ and lines from $\mathcal{L}$.
That is, $I(P,\mathcal{L})$ is the set of ordered pairs $(p,L)$ such that $p \in P$, $L \in \mathcal{L}$, and $p \in L$.
We write $G(P,\mathcal{L})$ to denote the \emph{incidence graph} for $P$ and $\mathcal{L}$.
This (oriented) graph has vertex set $P \cup \mathcal{L}$ and edge set $I(P,\mathcal{L})$.
More precisely, a pair $(p,L)$ with $p \in P$ and $L \in \mathcal{L}$ is an edge of $G(P,\mathcal{L})$ if and only if $p \in L$.

The celebrated \emph{Szemer\'edi--Trotter theorem}~\cite{szemTro83} states that every set $P$ of $m$ points and every set $\mathcal{L}$ of $n$ lines in the plane satisfy
\[|I(P,\mathcal{L})| \in O((mn)^{2/3}+m+n).\]
This bound is asymptotically tight, which can be seen by taking the $\sqrt{n} \times \sqrt{n}$ integer lattice
and bundles of parallel ``rich'' lines; see~\cite{pachAga95}.
In this paper, we focus on bounding the maximum number of point-line incidences in configurations $(P,\mathcal{L})$ that have some fixed forbidden induced subgraph in $G(P,\mathcal{L})$.
The study of such problems was initiated by Solymosi~\cite{solymosi06} and attracted considerable attention recently~\cite{mirSuk21,mirSukVer19,tomSuk21,sukZeng23}.

Let $P_1$ and $P_2$ be two sets of points in the plane and $\mathcal{L}_1$ and $\mathcal{L}_2$ be two sets of lines in the plane.
We say that $(P_1,\mathcal{L}_1)$ and $(P_2,\mathcal{L}_2)$ are \emph{isomorphic} if the graphs $G(P_1,\mathcal{L}_1)$ and $G(P_2,\mathcal{L}_2)$ are isomorphic.

Solymosi posed the following conjecture, which can be found in the book by Brass, Moser, and Pach~\cite[p.~291]{bmp05}.

\begin{conjecture}[\cite{bmp05}]
\label{conj-Solymosi}
For any set of points $P_0$ and for any set of lines $\mathcal{L}_0$ in the plane, the maximum number of incidences between $n$ points and $n$ lines in the plane containing no subconfiguration isomorphic to $(P_0,\mathcal{L}_0)$ is $o(n^{4/3})$.
\end{conjecture}

Solymosi~\cite{solymosi06} proved this conjecture in the special case that $P_0$ is a fixed set of
$k$ points in the plane in \emph{general position}, that is, no three points from $P_0$ lie on a common line, and $\mathcal{L}_0$ is the set of all lines determined by points from $P_0$.
Such configuration $(P_0,\mathcal{L}_0)$ is called a \emph{$k$-clique}. 
Mirzaei and Suk~\cite{mirSuk21} proved the conjecture for point sets that do not contain grids. 
In particular, they proved the following result.

\begin{theorem}[\cite{mirSuk21}]
\label{thm-mirsuk}
For a fixed $t\geq 2$ let $\mathcal{L}_a$ and $\mathcal{L}_b$ be two sets of $t$ lines in the plane, and let $P_0=\{\ell_a\cap \ell_b : \ell_a\in \mathcal{L}_a, \ell_b\in \mathcal{L}_b\}$ such that $|P_0|=t^2$. Then there is a constant $c=c(t)$ such that any arrangement of $m$ points and $n$ lines in the plane that does not contain a subconfiguration isomorphic to $(P_0,\mathcal{L}_a\cup \mathcal{L}_b)$ determines at most \[c(m^{\frac{2t-2}{3t-2}}n^{\frac{2t-1}{3t-2}}+m^{1+\frac{1}{6t-3}}+n)\] incidences.
\end{theorem}

This case was also considered by Suk and Tomon~\cite{tomSuk21} who also posed a special variant of Conjecture~\ref{conj-Solymosi} stating that for every fixed $k \geq 3$, every set of $n$ points and $n$ lines in the plane that do not contain a $k$-fan determines at most $o(n^{4/3})$ incidences.
Here, a \emph{$k$-fan} consists of $k+1$ points and $k+1$ lines such that $k$ points lie on a single line and the remaining $k$ lines connect them to the $(k+1)$st point.

Very recently, Mirzaei and Suk~\cite{mirSuk21} posed the following strengthening of Conjecture~\ref{conj-Solymosi}, whose statement is also mentioned by Brass, Moser, and Pach~\cite[p.~291]{bmp05}  for the configurations considered by Solymosi~\cite{solymosi06}.

\begin{conjecture}[\cite{mirSuk21}]
\label{conj-Suk}
For any set of points $P_0$ and for any set of lines $\mathcal{L}_0$ in the plane, there is a constant $\varepsilon = \varepsilon(P_0,\mathcal{L}_0)>0$ such that the maximum number of incidences between $n$ points and $n$ lines in the plane containing no subconfiguration isomorphic to $(P_0,\mathcal{L}_0)$ is $O(n^{4/3-\varepsilon})$.
\end{conjecture}

\section{Our results}

Here, we refute both Conjecture~\ref{conj-Solymosi} and Conjecture~\ref{conj-Suk} by finding sets~$P_0$ of $2k$ points and sets $\mathcal{L}_0$ of $\binom{k}{2}+1$ lines such that any configuration of $n$ \emph{lattice} points and $n$ lines does not contain subconfiguration isomorphic to $(P_0,\mathcal{L}_0)$.
We call the configuration $(P_0,\mathcal{L}_0)$ an \emph{extended regular $k$-gon} (see Section~\ref{sec-gon} for the definition).
In fact, our result is more general and shows this for point sets where coordinates are algebraic numbers of bounded degree.

For a positive integer $d$, we use $V_d$ to denote the set of all points from $\mathbb{R}^2$ whose coordinates are algebraic numbers of degree at most $d$.
In particular, we have $V_1=\mathbb{Q}^2$.

\begin{theorem}
\label{thm-counterexample}
For every $d \in \mathbb{N}$, there is a $k=k(d) \in \mathbb{N}$ such that for all $m,n \in \mathbb{N}$, if $(P_0,\mathcal{L}_0)$ is an image of the extended regular $k$-gon via a projective transformation, then for every set $P$ of $m$ points from~$V_d$, and for each set $\mathcal{L}$ of $n$ lines in the plane, the graph $G(P_0,\mathcal{L}_0)$ is not a subgraph of $G(P,\mathcal{L})$.

Moreover, if $d=1$, then $G(P_0,\mathcal{L}_0)$ is not a subgraph of $G(P,\mathcal{L})$ if and only if $k \notin \{3,4,6\}$.
\end{theorem}

Since there are configurations $(P,\mathcal{L})$ of $n$ \emph{lattice} points and $n$ lines with $\Omega(n^{4/3})$ incidences, Theorem~\ref{thm-counterexample} gives a counterexample to Conjectures~\ref{conj-Solymosi} and~\ref{conj-Suk} by setting $d=1$.
This includes known configurations $(P,\mathcal{L})$ found by Erd\H{o}s (see~\cite{pachAga95}) and Elekes~\cite{elekes02} and their generalization found by Sheffer and Silier~\cite{shefSil24}; see~\cite{bsr24} for their analysis.

On the other hand, Guth and Silier~\cite{guthSil23} recently discovered sharp examples for the Szemer\'{e}di--Trotter theorem that are not based on a rectangular lattice.
Their examples use points from the set $S = \{x+y\sqrt{m} \colon x,y \in \mathbb{Z}\}^2$ where $m$ is a non-square integer.
Since $S \subseteq V_2$, it follows from Theorem~\ref{thm-counterexample} that their construction combined with an extended regular $k$-gon with $k=k(2)$ gives a counterexample to Conjectures~\ref{conj-Solymosi} and~\ref{conj-Suk} as well.

Note that the forbidden subconfiguration $G(P_0,\mathcal{L}_0)$ is excluded as a subgraph of $G(P,\mathcal{L})$ and not necessarily as an induced subgraph of $G(P,\mathcal{L})$. 
Thus, we can, for example, add lines to $\mathcal{L}_0$ so that we include all lines determined by points in $P_0$, similarly as in the case proved by Solymosi~\cite{solymosi06}, and still have $|I(P,\mathcal{L})| \in \Omega(n^{4/3})$.

In the spirit of Székely's proof~\cite{Szekely} of the Szemer\'{e}di--Trotter theorem using the Crossing lemma, we introduce a new approach for proving the bound $O(n^{4/3-\varepsilon})$ on the number of incidences in configurations $(P,\mathcal{L})$ that do not contain certain fixed subconfigurations $(P_0,\mathcal{L}_0)$.

We first illustrate this method by considering the configurations from Theorem~\ref{thm-mirsuk}.
In this setting, our approach yields a simpler proof that gives the following slightly better bound for some of the terms.

\begin{theorem}
\label{thm-grids}
For an integer $t \geq 2$, let $\mathcal{L}_a$ and $\mathcal{L}_b$ be two sets of $t$ lines in the plane, and let $P_0=\{\ell_a\cap \ell_b : \ell_a\in \mathcal{L}_a, \ell_b\in \mathcal{L}_b\}$ such that $|P_0|=t^2$.
Then there is a constant $c=c(t)$ such that any arrangement of $m$ points and $n$ lines in the plane that does not contain a subconfiguration isomorphic to $(P_0,\mathcal{L}_a\cup \mathcal{L}_b)$ determines at most 
\[c(m^{\frac{2t-2}{3t-2}} n^{\frac{2t-1}{3t-2}}+m+n)\] 
incidences. 
\end{theorem}

For $m=n$, we obtain the bound $c n^{\frac{4}{3}-\frac{1}{9t-6}}$ on the number of incidences, which is the same as the bound obtained by Mirzaei and Suk~\cite{mirSuk21} in this case. 

Although our Crossing-lemma-based argument is not strong enough to improve the bound $o(n^{4/3})$ to $O(n^{4/3-\varepsilon})$ for $k$-cliques nor $k$-fans, it is sufficient to obtain such a bound for subdivided $k$-cliques.

For an integer $k \geq 3$ let $P_0$ be a set of $k+\binom{k}{2}$ points, where $k$ points are called \emph{black} and the remaining $\binom{k}{2}$ points are \emph{white}, and let $\mathcal{L}_0$ be a set of $2\binom{k}{2}$ lines in the plane.
We call the configuration $(P_0,\mathcal{L}_0)$ a \emph{subdivided $k$-clique} if, for any two black points, there are two lines from $\mathcal{L}_0$, each containing one of the black points, that intersect in a white point, each white point lies on exactly two lines from $\mathcal{L}_0$, and each line from $\mathcal{L}_0$ contains exactly one black and one white vertex.
That is, if two points from $P_0$ are connected by an edge if and only if they lie on a line from $\mathcal{L}_0$, then the resulting graph is a 1-subdivision of $K_k$.

\begin{theorem}
\label{thm-subdivision}
For an integer $k \geq 3$, there are constants $c=c(k)$ and $\delta=\delta(k)$ such that any arrangement of $m$ points and $n$ lines in the plane that does not contain a subconfiguration isomorphic to a subdivided $k$-clique determines at most 
\[c(m^{3/4-\delta}n^{1/2-\delta}+m\log^2{m}+n)\] 
incidences. 
\end{theorem}

By setting $m=n$, we immediately obtain the following bound.

\begin{corollary}
\label{cor-subdivision}
For an integer $k \geq 3$, there are constants $c=c(k)$ and $\delta=\delta(k)$ such that any arrangement of $n$ points and $n$ lines in the plane that does not contain a subconfiguration isomorphic to a subdivided $k$-clique determines at most $c n^{5/4-\delta}$ incidences. 
\end{corollary}

For subdivided $k$-cliques, we also apply a method by Suk and Tomon~\cite{tomSuk21} to provide a lower bound that is fairly close to the upper bound from Corollary~\ref{cor-subdivision}.

\begin{theorem}
\label{thm-lowerbound}
For all positive integers $n$ and $k \geq 3$, there exists a point-line configuration $(P,\mathcal{L})$ such that $|P|=|\mathcal{L}|=n$ and 
\[|I(P,\mathcal{L})|\geq n^{\frac{5}{4}-\frac{1}{2k}+o(1)},\] and the incidence graph of $(P,\mathcal{L})$ does not contain a subconfiguration isomorphic to a subdivided $k$-clique.
\end{theorem}

\paragraph{Open problems}

It is possible that our Crossing-lemma-based approach applies to other forbidden configurations such as cycles in the incidence graph, which is a motivation for future work.

The smallest sets $P_0$ and $\mathcal{L}_0$ that we found such that there are sets of $n$ points and $n$ lines in the plane with no copy of $G(P_0,\mathcal{L}_0)$ in their incidence graph that determine $\Theta(n^{4/3})$ incidences, contain 10 points and 11 lines, respectively.
Although we did not try to find minimum counterexamples, it might be interesting to see what the smallest such configuration $(P_0,\mathcal{L}_0)$ is.

Even though Conjectures~\ref{conj-Solymosi} and~\ref{conj-Suk} do not hold in general, a natural problem is to determine classes of configurations $(P_0,\mathcal{L}_0)$ for which the maximum number of incidences is $o(n^{4/3})$.
In particular, deciding whether forbidding a 3-fan reduces the maximum number of incidences to $o(n^{4/3})$ remains open.

\paragraph{Note added}
After finishing this paper, we learned that, recently, Solymosi~\cite{solymosi} independently disproved Conjectures~\ref{conj-Solymosi} and~\ref{conj-Suk} in an unpublished manuscript.
His idea uses configurations that are not embeddable in the plane with rational coordinates.
In particular, Solymosi uses an arrangement of nine points and nine lines found by Perles (see~\cite{ziegler08} and~\cite{grun03}), which is by one point smaller than our smallest counterexample for $d=1$.

\section{Proof of Theorem~\ref{thm-counterexample}}
\label{sec-gon}

To prove Theorem~\ref{thm-counterexample}, we use a variant of the well-known and easy-to-prove fact that regular $k$-gons with $k \neq 4$ cannot be subsets of integer lattices.
For $k \geq 3$, we will show that by adding $k$ additional points representing the slopes of the lines determined by points of a regular $k$-gon and by adding $\binom{k}{2}+1$ lines to regular $k$-gons, the incidence graph of the resulting configuration $(P_0,\mathcal{L}_0)$ determines~$P_0$ up to a projective transformation.
This is because, by a result of Jamison~\cite{jamison86}, any set of $k$ points in general position determining only $k$ slopes is an affine image of a regular $k$-gon.
We show that, by choosing suitable $k$, the set $V_d$ does not contain an image of the resulting extended regular $k$-gon formed by $2k$ points via any projective transformation.
We then finish the proof by showing that this is true for $V_1$ if and only if $k \notin \{3,4,6\}$.

For a positive integer $n$ and a set $S$, we use $[n]$ to denote the set $\{1,\dots,n\}$ and write $\binom{S}{n}$ for the set of all unordered $n$-tuples of distinct elements from~$S$.
If $p$ and $q$ are distinct points, then we use $pq$ and $\overline{pq}$ to denote the line segment and the line, respectively, determined by $p$ and $q$.

For an integer $k \geq 3$, we start by defining the following bipartite graph~$H_k$, which, as we will prove later, can be realized as an incidence graph $G(P_0,\mathcal{L}_0)$ for some $P_0$ and $\mathcal{L}_0$; see Figure~\ref{fig-extendedRegular} for an illustration.
We let the vertex set of $H_k$ be
\[\{v_1,\dots,v_k,t_1,\dots,t_k\} \cup \left\{L_{i,j} \colon \{i,j\} \in \binom{[k]}{2}\right\}\cup\{L_\infty\}.\]
The edges of $H_k$ are the following pairs (where indices are taken modulo $k$):
\begin{enumerate}[(a)]
    \item $(v_i,L_{i,j})$ and $(v_j,L_{i,j})$ for every $\{i,j\} \in \binom{[k]}{2}$,
    \item $(t_{2i-1},L_{i-s,i+s+1})$ for all $i \in [\lceil k/2 \rceil]$ and $s \in \{0,\dots,\lfloor k/2\rfloor -1\}$,
    \item $(t_{2i},L_{i-s,i+s+2})$ for all $i \in [\lfloor k/2\rfloor]$ and $s \in \{0,\dots,\lceil k/2\rceil-2\}$,
    \item $(t_i, L_\infty)$ for every $i \in [k]$.
\end{enumerate} 

We now prove that $H_k$ can be realized as an incidence graph. 

\begin{figure}[ht]
    \centering
    \includegraphics{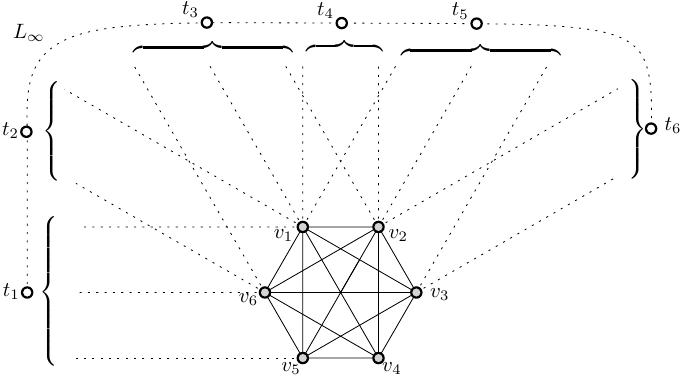}
    \caption{Realizing the graph $H_6$. The lines $L_{i,j}$ correspond to the lines determined by points $v_i$ and $v_j$.}
    \label{fig-extendedRegular}
\end{figure}

\begin{lemma}
\label{lem-realization}
For each integer $k \geq 3$, there is a set of points $P_0$ and a set of lines~$\mathcal{L}_0$ in the plane such that $H_k= G(P_0,\mathcal{L}_0)$.
\end{lemma}
\begin{proof}
See Figure~\ref{fig-extendedRegular} for an illustration.
Let $R$ be the set of vertices $v_1,\dots,v_k$ of a regular $k$-gon in the real projective plane.
The points in $R$ determine exactly $k$ slopes, and we let $t_1,\dots,t_k$ be the points on the line at infinity such that every line determined by two points from $R$ contains some point $t_i$.
We call the resulting set $Q$ of $2k$ points $v_1,\dots,v_k,t_1,\dots,t_k$ in the real projective plane an \emph{extended regular $k$-gon}.
Note that $H_k$ is the incidence graph of $Q$ together with the set containing all lines $\overline{v_iv_j}$ and the line $L_{\infty}$ in infinity.

After applying a projective transformation $h$ that sends all points $v_1,\dots,v_k,\allowbreak t_1,\dots,t_k$ to the real plane, we obtain the point set $P_0 = h(Q)$ in the plane. 
By letting $\mathcal{L}_0$ to be the set of lines $\overline{h(v_i)h(v_j)}$ together with the line $h(L_\infty)$ containing the points $h(t_1),\dots,h(t_k)$, we get $H_k = G(P_0,\mathcal{L}_0)$ since $h$ preserves point-line incidences.
\end{proof}

The following lemma implies that this realization of $H_k$ is unique up to a projective transformation.

\begin{lemma}
\label{lem-realizationUnique}
For every integer $k \geq 3$, if $P_0$ is a set of points and $\mathcal{L}_0$ is a set of lines in the plane such that $H_k=G(P_0,\mathcal{L}_0)$, then $P_0$ is an image of an extended regular $k$-gon via a projective transformation.
\end{lemma}

To prove this lemma, we use the following result of Jamison~\cite{jamison86}, who proved that the affine images of regular $k$-gons are the only sets of $k$ points in general position determining exactly $k$ slopes.

\begin{theorem}[\cite{jamison86}]
\label{thm-directions}
Any set of $k \geq 3$ points in the plane, in general position, that determines
exactly $k$ slopes, is affinely equivalent to $k$ of the vertices of a regular $k$-gon.
\end{theorem}

We can now proceed with the proof of Lemma~\ref{lem-realizationUnique}.

\begin{proof}[Proof of Lemma~\ref{lem-realizationUnique}]
Assume that a set $P_0$ of points and a set $\mathcal{L}_0$ of lines in the plane satisfy $H_k=(P_0,\mathcal{L}_0)$.
Let $Q$ be the set containing points from $P_0$ that correspond to $v_1,\dots,v_k$ in $H_k$.
The set $Q$ is in general position, as $\mathcal{L}_0$ contains all lines determined by pairs of points from $Q$ and each of them contains exactly two points of $Q$.
Let $h$ be a projective transformation that sends the line from~$\mathcal{L}_0$ corresponding to $L_\infty$ in $H_k$ to infinity. 
Then, the points from $h(Q)$ are still in general position and determine exactly $k$ slopes as every line determined by points from $h(Q)$ contains one of the points $t_i$, and each point $t_i$ lies on some line determined by points from $h(Q)$.

By Theorem~\ref{thm-directions}, there is an affine transformation $g$ such that $g(h(Q))$ is a regular $k$-gon $R$.
Let $R'$ be the extended regular $k$-gon obtained from $R$.
By considering the projective transformation $f = g \circ h$ and using the fact that projective transformations preserve point-line incidences, we see that $P_0 = f(R')$ and thus $P_0$ is an image of an extended regular $k$-gon via a projective transformation.
\end{proof}

It is well-known that regular $k$-gons can have all vertices with integer coordinates if and only if $k = 4$.
Here, we prove that, for some values of~$k$, images of extended regular $k$-gons via projective transformation cannot be embedded in subsets of the plane formed by points with coordinates of small degree.

For an algebraic number $x$, we use $deg(x)$ to denote the degree of $x$ over $\mathbb{Q}$.
We note that for $r \in \mathbb{Q}$ the number $\cos(r \pi)$ is algebraic; see~\cite{pss24}, for example.

\begin{lemma}
\label{lem-embedding}
There is an absolute constant $c \geq 1$ such that, for all integers $d \geq 1$ and $k \geq 5$ satisfying
\begin{equation}
\label{eq-condition}
deg\left(\left(1+2\cos\left(\frac{2\pi}{k}\right)\right)^2 \right) > d^c
\end{equation}
and for every projective transformation $f$, if $R'$ is an extended regular $k$-gon, then $f(R')$ is not a subset of $V_d$.
\end{lemma}

For $d=1$, we later show that the condition~\eqref{eq-condition} on $k$ in Lemma~\ref{lem-embedding} is satisfied for any $k \notin \{3,4,6\}$.
For example, by setting $k=5$, we obtain $\left(1+2\cos\left(\frac{2\pi}{5}\right)\right)^2 = \frac{3+\sqrt{5}}{2} \notin \mathbb{Q}$.
On the other hand, choosing $k=6$ gives $\left(1+2\cos\left(\frac{\pi}{3}\right)\right)^2 = 4 \in \mathbb{Q}$.
In fact, there is a projective transformation $f$ such that $f(R') \subseteq \mathbb{Z}^2$ if $k \in \{3,4,6\}$; see the proof of Theorem~\ref{thm-counterexample}.

Before proving Lemma~\ref{lem-embedding}, we first state some auxiliary definitions.
If $p$ and $q$ are two points in the real projective plane, then we denote their distance by $|pq|$. 
If $a$, $b$, $c$ are three distinct points on a line, then their \emph{ratio} is $(a, b; c) = \xi\frac{|ac|}{|bc|}$ where $\xi=1$ if the rays from $a$ to $c$ and from $b$ to $c$ point in the same direction, and $\xi=-1$ otherwise.
For four distinct collinear points $a,b,c,d$, their \emph{cross-ratio} is defined as
$(a,b;c,d) = \frac{(a,b;c)}{(a,b;d)}$. 
It is well-known that the cross-ratio is invariant under projective transformations.
If $a,b,c,d$ appear on the line in this order along the orientation of the line, then
\[(a,b;c,d) = \frac{|ac||bd|}{|ad||bc|}.\]
If one of the four points is the line's point at infinity, then the two distances involving that point are removed from the formula.

\begin{proof}[Proof of Lemma~\ref{lem-embedding}]
For $d \geq 1$, let $R'$ be an extended regular $k$-gon with $k \geq 5$, and let $f$ be a projective transformation.
We may assume that $f(R') \subseteq \mathbb{R}^2$ as otherwise some points of $f(R')$ are mapped to the line in infinity and are not contained in~$V_d$.
As before, we use $R$ to denote the regular $k$-gon on vertices $v_1,\dots,v_k$ that is a subset of $R'$.
We let $w$ be the intersection point of the lines $\overline{v_1v_2}$ and $\overline{v_{k-1},v_k}$.
Since $k \geq 5$, the point $w$ is not one of the points from $\{v_1,\dots,v_k\}$ nor does it lie in infinity.

Let $s$ be the length of each side of $R$, that is, $s=|v_iv_{i+1}|$ for every $i \in [k]$.
We also use $r$ to denote the distance $|wv_1|$.
The points $w$, $v_1$, and $v_k$ span an isosceles triangle with sides of lengths $r$, $r$, and $s$.
Now, the angles $wv_1v_k$ and $wv_kv_1$ are equal to $(2\pi - 2\frac{(k-2)\pi}{k})/2=2\pi/k$, since the internal angle of $R$ is $\frac{(k-2)\pi}{k}$.
It follows that the angle $v_1wv_k$ equals $\pi - 2 \cdot 2\pi/k = \frac{(k-4)\pi}{k}$.
By the \emph{Law of sines}, we get
\begin{equation}
\label{eq-ratio}
\frac{s}{r} = \frac{\sin\left(\frac{(k-4)\pi}{k}\right)}{\sin\left(\frac{2\pi}{k}\right)} = \frac{\sin\left(\frac{4\pi}{k}\right)}{\sin\left(\frac{2\pi}{k}\right)} = 2\cos\left(\frac{2\pi}{k}\right),
\end{equation}
where we also used the \emph{Double angle identity} $\sin(2x) = 2\sin(x)\cos(x)$.

To compute the cross-ratio, note that, by~\eqref{eq-ratio},
\begin{equation}
\label{eq-crossRatio}
(t_1,w;v_1,v_2) = \frac{|t_1v_1||wv_2|}{|t_1v_2||wv_1|} = \frac{|wv_2|}{|wv_1|} = \frac{r+s}{r} = 1+2\cos\left(\frac{2\pi}{k}\right),
\end{equation}
where we used the fact that $t_1$ is the line's point at infinity.

Now, suppose for contradiction that $R'$ is an extended regular $k$-gon satisfying $f(R') \subseteq V_d$ and $deg\left(\left(1+2\cos\left(\frac{2\pi}{k}\right)\right)^2\right) > d^c$.

We show that the point $w$ has coordinates of degree at most $d^{c'}$ for some absolute constant $c' \geq 1$.
The point $w$ is the intersection point of two lines $\overline{v_1v_2}$ and $\overline{v_{k-1},v_k}$. 
Thus, the coordinates of $w$ can be expressed in terms of the coordinates of the points $v_1,v_2,v_{k-1},v_k$ using a fixed number of additions, multiplications, and divisions.
If $\alpha$ and $\beta$ are two algebraic numbers of degrees $m$ and $n$, respectively, then the degree of $\alpha + \beta$, $\alpha \cdot \beta$, and $\alpha^{-1}$ are at most $mn$, $mn$, and $m$, respectively.
Since $v_1,v_2,v_{k-1},v_k \in V_d$, the coordinates of all points $v_1,v_2,v_{k-1},v_k$ have degree at most $d$.
It follows that the degrees of the coordinates of $w$ are at most $d^{c'}$ for some absolute constant $c' \geq 1$.

Using the bounds on degrees of numbers obtained by additions and multiplications and the facts $f(R') \subseteq V_d$ and $w \in V_{d^c}$, the squares of all distances between two points from $f(R' \cup \{w\})$ are numbers of degree at most $d^c$ for some absolute constant $c \geq 1$.
Thus, the number $(f(t_1),f(w);f(v_1),f(v_2))^2$ has degree at most $d^c$.
Since projective transformations preserve cross-ratios, we have
\begin{equation}
\label{eq-preserving}
(t_1,w;v_1,v_2) = (f(t_1),f(w);f(v_1),f(v_2)).
\end{equation}
On the other hand, the equation~\eqref{eq-crossRatio} 
gives
\[
deg\left((t_1,w;v_1,v_2)^2\right) = deg\left(\left(1+2\cos\left(\frac{2\pi}{k}\right)\right)^2\right) > d^c.
\]
By combining this with~\eqref{eq-preserving}, we obtain that $(f(t_1),f(w);f(v_1),f(v_2))^2$ has degree larger than $d^c$, a contradiction.
\end{proof}

We now characterize values of $k \geq 5$ that satisfy~\eqref{eq-condition} for $d=1$ by showing that $k=6$ is the only exception violating~\eqref{eq-condition}.
To do so, we use the following classical result called \emph{Niven's theorem}~\cite{niven56}; see also~\cite{pss24}.

\begin{theorem}[Niven's theorem~\cite{niven56}]
\label{thm-niven}
If $r$ and $\cos(r\pi)$ are both rational, then \[\cos(r\pi)\in\left\{0,\pm 1,\frac{\pm 1}{2}\right\}.\]
\end{theorem}

In particular, if $r\pi \in (0,\pi/2)$ and both $r$ and $\cos(r\pi)$ are rational, then Niven's theorem implies $r =\frac{1}{3}$.
We also use the following extension of Niven’s theorem to quadratic number fields, which can be found in the paper by Panraksa, Samart, and Sriwongsa~\cite{pss24}, for example.

\begin{theorem}[\cite{pss24}]
\label{thm-nivenQuadratic}
Let  $r\in \mathbb{Q}$.
If  $\cos(r\pi)$ is a quadratic irrational, then \[\cos(r\pi) \in\left\{\frac{\pm \sqrt{2}}{2},\frac{\pm \sqrt{3}}{2},\frac{\pm 1 \pm \sqrt{5}}{4}\right\}.\]
\end{theorem}

Now, if $r\pi \in (0,\pi/2)$, $r$ is rational, and $\cos(r\pi)$ is a quadratic irrational, then Theorem~\ref{thm-nivenQuadratic} implies $r \in \left\{\frac{1}{4},\frac{1}{6},\frac{1}{5},\frac{2}{5}\right\}$.

Lehmer~\cite{lehmer33} proved the following general result; see also~\cite{pss24}.

\begin{theorem}[\cite{lehmer33}]
\label{thm-lehmer}
Let $m, n \in \mathbb{Z}$, with $n > 2$, be relatively prime.
Then $\cos(2\pi m/n)$ is an algebraic
number of degree $\varphi(n)/2$, where $\varphi(n)$ is the Euler’s totient function.
\end{theorem}

We now characterize the values of $k$ that satisfy~\eqref{eq-condition} for $d=1$.

\begin{lemma}
\label{lem-classification}
For every integer $k \geq 5$ with $k \neq 6$, we have $\left(1+2\cos\left(\frac{2\pi}{k}\right)\right)^2 \notin \mathbb{Q}$.
\end{lemma}
\begin{proof}
Let $x=\cos\left(\frac{2\pi}{k}\right)$ and note that $0< \frac{2\pi}{k} < \frac{\pi}{2}$ and $x>0$, since $k \geq 5$.
Assume that $(1+2x)^2 = r$ for some rational number $r$.
We show that $k=6$.
After rewriting, we get $4x^2+4x+1-r=0$, so $x$ is a solution to a quadratic equation with rational coefficients.
In other words, $x$ is either a rational number or a real quadratic irrational.

If $x$ is a rational number, then Niven's theorem (Theorem~\ref{thm-niven}) implies $k = 6$, where we used $0<\frac{2\pi}{k}<\frac{\pi}{2}$.
Similarly, if $x$ is a quadratic irrational, then Theorem~\ref{thm-nivenQuadratic} gives $k \in \{5,8,10,12\}$, where we again used $0<\frac{2\pi}{k}<\frac{\pi}{2}$.
However, for these four values of $k$, the expression $(1+2x)^2$ becomes $\frac{3+\sqrt{5}}{2}$, $3+2\sqrt{2}$, $\frac{7}{2} +\frac{3\sqrt{5}}{2}$, and $4+2\sqrt{3}$, respectively, all of which are irrational numbers.
Thus, we indeed have $k=6$.
\end{proof}

We prove a similar result for numbers with a higher degree.

\begin{lemma}
\label{lem-higherDegree}
For every positive integer $d$, there is a positive integer $k = k(d)$ such that $\left(1+2\cos\left(\frac{2\pi}{k}\right)\right)^2$ has degree larger than $d$.
\end{lemma}
\begin{proof}
We let $p \geq 3$ be a prime number such that $(p+1)/4 > d^c$.
Let $x = \cos\left(\frac{2\pi}{p}\right)$.
By Theorem~\ref{thm-lehmer}, $x$ is an algebraic number of degree $D=\varphi(p)/2=(p-1)/2$.
Suppose, for contradiction, that the degree of $y=(1+2x)^2$ is at most $(D-1)/2$.
That is, there is a polynomial $P$ with rational coefficients and degree at most $(D-1)/2$ such that $P(y)=0$.
Then, however, the polynomial $Q(x) = P((1+2x)^2)$ satisfies $Q(x)=P(y)=0$, while $Q$ has rational coefficients and degree at most $2 \cdot ((D-1)/2)=D-1$.
This contradicts the fact that the degree of $x$ is $D$.

Altogether, we see that the degree of $y$ is at least 
\[\frac{D-1}{2}+1 = \frac{(p-1)/2-1}{2}+1 = \frac{p+1}{4} > d^c,\]
which finishes the proof by setting $k=k(d)=p$.
\end{proof}

Now, we are ready to put everything together and prove Theorem~\ref{thm-counterexample}.

\begin{proof}[Proof of Theorem~\ref{thm-counterexample}]
For a positive integer $d$, let $P$ be a set of $n$ points from $V_d$ and $\mathcal{L}$ a set of $m$ lines in the plane.
We choose an integer $k$ satisfying~\eqref{eq-condition}.
Such an integer exists by Lemma~\ref{lem-higherDegree}.
For $d=1$, Lemma~\ref{lem-classification} implies that all values of $k \geq 5$ besides 6 satisfy this condition.
Let $P_0$ be a set of points and $\mathcal{L}_0$ a set of lines in the plane such that $H_k=G(P_0,\mathcal{L}_0)$.
Such sets exist by Lemma~\ref{lem-realization}.

Suppose for contradiction that $H_k$ is a subgraph of $G(P,\mathcal{L})$.
In particular, we have $P_0 \subseteq P$ as the edges of $G(P,\mathcal{L})$ are oriented from $P$ to $\mathcal{L}$.
By Lemma~\ref{lem-realizationUnique}, there is a projective transformation $f$ and an extended regular $k$-gon $R'$ such that $P_0 = f(R')$.
Since $k$ satisfies~\eqref{eq-condition}, Lemma~\ref{lem-embedding} gives $P_0=f(R') \not\subseteq V_d$.
However, since $P \subseteq V_d$, we obtain a contradiction with $P_0 \subseteq P$.

In the case $d=1$, it remains to show that for $k \in \{3,4,6\}$, there is a projective transformation that maps points of an extended regular $k$-gon to $\mathbb{Z}^2$, completing the characterization.
For vertices $v_1=[\frac{-1}{2}:\frac{\sqrt{3}}{2}:1]$, $v_2=[\frac{1}{2}:\frac{\sqrt{3}}{2}:1]$, $v_3=[1:0:1]$, $v_4=[\frac{1}{2}:\frac{-\sqrt{3}}{2}:1]$, $v_5=[\frac{-1}{2}:\frac{-\sqrt{3}}{2}:1]$, and $v_6=[-1:0:1]$ of a regular $6$-gon $R'$ written in homogeneous coordinates, it suffices to consider the non-singular matrix \[\begin{pmatrix}10 & 20\sqrt{3} & 40\\10 & 20\sqrt{3}& 20 \\ 2 & \sqrt{3} & 0\end{pmatrix}.\]
The corresponding projective transformation $f$ then gives $f(v_1)=[130:90:1]$, $f(v_2)=[30:22:1]$, $f(v_3)=[25:15:1]$, $f(v_4)=[-30:10:1]$, $f(v_5)=[-2:3:1]$, and $f(v_6)=[-15:-5:1]$.
For the points $t_1=[1:0:0]$, $t_2=[-\sqrt{3}:1:0]$, $t_3=[\frac{-1}{\sqrt{3}}:1:0]$, $t_4=[0:1:0]$, $t_5=[\frac{1}{\sqrt{3}}:1:0]$, and $t_6=[\sqrt{3}:1:0]$, we get $f(t_1)=[5:5:1]$, $f(t_2)=[-10:-10:1]$, $f(t_3)=[50:50:0]$, $f(t_4)=[20:20:1]$, $f(t_5)=[14:14:1]$, and $f(t_6)=[10:10:1]$.
Thus, $f(R') \subseteq \mathbb{Z}^2$.

Since points $v_1$, $v_3$, and $v_5$ induce a regular $3$-gon, we also get extended regular $3$-gon as a subset of $\mathbb{Z}^2$.
Similarly, points $v_1$, $v_2$, $v_4$, and $v_6$ induce an affine image of a regular $4$-gon, so by composing this affine map with $f$ we get even an extended regular $4$-gon as a subset of $\mathbb{Z}^2$.
\end{proof}

\section{Proof of Theorem~\ref{thm-grids}}

Let $\mathcal{L}_a$ and $\mathcal{L}_b$ be two sets of $t$ lines in the plane, and let $P_0=\{\ell_a\cap \ell_b : \ell_a\in \mathcal{L}_a, \ell_b\in \mathcal{L}_b\}$ such that $|P_0|=t^2$.
Here, we prove the upper bound 
\[|I(P,\mathcal{L})| \in O(m^{\frac{2t-2}{3t-2}} n^{\frac{2t-1}{3t-2}}+m+n)\] 
for all configurations $(P,\mathcal{L})$, where $|P|=m$ and $|\mathcal{L}|=n$, that do not contain the subconfiguration $(P_0,\mathcal{L}_a\cup \mathcal{L}_b)$.
To do so, we use a new approach based on a variant of the celebrated Crossing lemma.
 
We recall the notion of the crossing number. 
A \emph{drawing} of a simple graph $G$ in the plane is a mapping $f$ that assigns to each vertex of $G$ a distinct point of the plane, and to each edge $uv$ of $G$ a continuous arc connecting $f(u)$ and $f(v)$, without passing through the image of any other vertex. 
The \emph{crossing number} $\textrm{cr}(G)$ of $G$ is the minimum number of edge crossing points in any drawing of~$G$. 
By a well-known result of Ajtai et al.~\cite{ACNS}, the crossing number of $G$ with $n$ vertices and $e$ edges is at least $\Omega(e^3/n^2)$, assuming that $e\in \Omega(n)$. 
By putting further restrictions on the graph, one can obtain improved lower bounds for the crossing number.

In our setting, we use the following variant of the Crossing lemma proved by Pach, Spencer, and T\'{o}th~\cite{pachspencertoth}. 
This is a special variant of their more general result with slightly worse bounds; see Theorem~\ref{thm:crossing_general}.

\begin{theorem}[\cite{pachspencertoth}]
\label{thm:crossing}
Let $G$ be a graph with $n$ vertices and $e\geq 4n$ edges, which does not contain a complete bipartite graph $K_{r,s}$ with $s \geq r$. Then the crossing number of $G$ satisfies \[\textrm{cr}(G)\geq c_{r,s}\frac{e^{3+1/(r-1)}}{n^{2+1/(r-1)}},\] for some constant $c_{r,s}>0$.
\end{theorem}

With Theorem~\ref{thm:crossing}, the rest of the argument is rather simple.

\begin{proof}[Proof of Theorem~\ref{thm-grids}]
Let $P$ be the set of $m$ points, $\mathcal{L}$ be the set of $n$ lines, and $I=|I(P,\mathcal{L})|$. 
We use $P'$ to denote the set of $n$ points dual to lines from $\mathcal{L}$ and $\mathcal{L}'$ to denote the set of $m$ lines dual to points from $P$.
Let $G$ be the graph whose vertex set is formed by the points of $P'$, and whose edges connect the $i$th and the $(i+1)$st vertex along each line from $\mathcal{L}'$ taken from the left, where $i$ is odd.
Then, $G$ has $n$ vertices and at least $(I-m)/2$ edges, and any two edges of $G$ on the same line from $\mathcal{L}'$ are disjoint. 

We assume that $(I-m)/2 \geq 4n$ and $(I-m)/2 \geq I/4$, as otherwise $I \in O(n)$ or $I \in O(m)$ and we are done by choosing $c$ large enough.
Note that $G$ does not contain $K_{t,t}$ as a subgraph, since a subgraph of $G$ isomorphic to $K_{t,t}$ corresponds to a subconfiguration isomorphic to $(P_0,\mathcal{L}_a\cup\mathcal{L}_b)$ in $(P,\mathcal{L})$.
This is because it follows from our choice of $G$ that there are no two edges of a copy of $K_{t,t}$ in $G$ lying on the same line from $\mathcal{L}'$ and then, by duality, the copy of $K_{t,t}$ in $G$ corresponds to two sets of $t$ lines from $\mathcal{L}$ where any two such lines share a point from $P$ and these intersections give $t^2$ distinct points from $P$.

Thus, Theorem~\ref{thm:crossing} for $r=t$ implies that 
\[\textrm{cr}(G)\geq c'\frac{((I-m)/2)^{3+1/(t-1)}}{n^{2+1/(t-1)}}\geq c'\frac{(I/4)^{3+1/(t-1)}}{n^{2+1/(t-1)}}\]
for some constant $c'=c'(t)>0$.
On the other hand, since any two lines intersect at most once, we also have $\textrm{cr}(G)\leq \binom{m}{2}$.
Combining these upper and lower bounds on $\textrm{cr}(G)$ implies the result.
\end{proof}

Note that it follows from the proof that the bound on $|I(P,\mathcal{L})|$ remains valid for any configuration $(P_0,\mathcal{L}_0)$ such that the graph $H$ corresponding to $G(P_0,\mathcal{L}_0)$ in $G$ is bipartite, as $H$ is then a subgraph of $K_{t,t}$ for some $t$.

\section{Proof of Theorem~\ref{thm-subdivision}}

In this section, we prove the upper bound
\[|I(P,\mathcal{L}) |\in O(m^{3/4-\delta}n^{1/2-\delta}+m\log^2{m}+n)\]
for some constant $\delta=\delta(k)>0$, all configurations $(P,\mathcal{L})$, where $|P|=m$ and $|\mathcal{L}|=n$, that do not contain a subdivided $k$-clique.
Our approach is the same as in the proof of Theorem~\ref{thm-grids}; we only use a different variant of the Crossing lemma, proved by Pach, Spencer, and T\'{o}th~\cite{pachspencertoth}. 

A graph property $\mathcal{P}$ is $\emph{monotone}$ if whenever a graph $G$ satisfies $\mathcal{P}$, then every subgraph of $G$ also satisfies $\mathcal{P}$, and whenever graphs $G_1$ and $G_2$ satisfy $\mathcal{P}$, then their disjoint union also satisfies $\mathcal{P}$.
We use $ex(n,\mathcal{P})$ to denote the maximum number of edges of an $n$-vertex graph satisfying $\mathcal{P}$.
If $\mathcal{P}$ is the property that our graph does not contain a subgraph isomorphic to a fixed graph $H$, then we write $ex(n,H)$ instead of $ex(n,\mathcal{P})$.

\begin{theorem}[\cite{pachspencertoth}]
\label{thm:crossing_general}
Let $\mathcal{P}$ be a monotone graph property with $ex(n,\mathcal{P}) \in O(n^{1+\alpha})$ for some $\alpha>0$. 
Then, there exists two constants $c,c'>0$ such that for any $n$ vertex graph $G$ with property $\mathcal{P}$ and with at least $e\geq cn\log^2 n$ edges we have
\[\textrm{cr}(G)\geq c'\frac{e^{2+1/\alpha}}{n^{1+1/\alpha}}.\]
\end{theorem}

We also use the following Tur\'{a}n-type bound by Janzer~\cite{janzer19}.

\begin{theorem}[\cite{janzer19}]
\label{thm-TuranMaxDegree}
Let $H$ be a subdivision of $K_k$ with $k \geq 3$.
Then, there exists a constant $c = c(k) > 0$ such that $ex(n, H) \leq cn^{3/2 - 1/(4k-6)}$.
\end{theorem}

The rest of the proof is similar to the proof of Theorem~\ref{thm-grids}, only enforcing the occurrence of the subdivided $k$-clique is more involved.

\begin{proof}[Proof of Theorem~\ref{thm-subdivision}]
Let $P$ be the set of $m$ points, $\mathcal{L}$ be the set of $n$ lines, and $I=|I(P,\mathcal{L})|$. 
Let $G$ be the graph whose vertex set is formed by the points of $P$ and whose edges connect the $i$th and the $(i+1)$st vertex along each line from $\mathcal{L}$ taken from the left, where $i$ is odd.
Then, $G$ has $m$ vertices and at least $(I-n)/2$ edges, and any two edges of $G$ on the same line from $\mathcal{L}$ are disjoint. 
We assume that $(I-n)/2 \in \Omega(m\log^2{m})$ and $(I-n)/2 \geq I/4$, as otherwise $I \in O(n)$ or $I \in O(m\log^2{m})$ and we are done by choosing $c$ large enough.

For integers $q \geq 2$, $m \geq 2$, and $n \geq 1$, we use $R_m(n;q)$ to denote the $q$-color Ramsey number of the complete $m$-uniform hypergraph $K^{(m)}_n$ on $n$ vertices.
That is, $R_m(n;q)$ is the smallest positive integer $N$ such that any $q$-coloring of the hyperedges of $K^{(m)}_N$ contains a monochromatic copy of $K^{(m)}_n$.
Such a number exists by Ramsey's theorem.

Let $R=R_4(2k;25)$ and $R' = R_3(2R;7)$. 
We show that if $G$ contains a 1-subdivision of $K_{R'}$ as a subgraph, then $(P,\mathcal{L})$ contains a subdivided $k$-clique as a subconfiguration.
Assume that there is such a 1-subdivision $H$ of $K_{R'}$ in $G$.
We label the black vertices of $H$ as $b_i$ with $i \in [R']$.
For two black vertices $b_i$ and $b_j$ in $H$, we use $w(b_i,b_j)$ to denote the white vertex that is the intersection of the two lines from $\mathcal{L}$ containing $b_i$ and $b_j$.

We color each triple $(b_{i_1},b_{i_2},b_{i_3})$ with $1 \leq i_1 < i_2 < i_3 \leq R'$ by one of $7=1+3!$ colors as follows.
For a permutation $\pi \in S_3$, we color $(b_{i_1},b_{i_2},b_{i_3})$ with $\pi$ if the line $\overline{b_{i_{\pi(1)}}w(b_{i_{\pi(1)}}b_{i_{\pi(2)}})}$ is in $\mathcal{L}$ and contains the point $b_{i_{\pi(3)}}$.
If there are more choices, we select the color arbitrarily.
If there is no such $\pi \in S_3$, then we color $(b_{i_1},b_{i_2},b_{i_3})$ with color $\emptyset$.
By the choice of $R'$, there is a $2R$-tuple of black vertices from $H$ with all triples monochromatic.
We let $H_1$ be the 1-subdivision of $K_{2R}$ that is a subgraph of $H$ and contains these $2R$ black vertices.

We show that there is a 1-subdivision $H_2$ of $K_R$ that is a subgraph of $H_1$ such that no line containing an edge of $H_2$ contains two black vertices of $H_2$.
First, assume that all triples of vertices from $H_1$ have color $\emptyset$.
Then, for any edge of $H_1$, the line containing the edge does not contain any other black point, as otherwise, we have a triple of black vertices from $H_1$ that does not have the color $\emptyset$.
We can then set $H_2$ as any 1-subdivision of $K_R$ that is a subgraph of~$H_1$.
Here, we used the fact that $G$ cannot contain two edges that share a vertex and lie on a common line from $\mathcal{L}$.
Now, assume that all triples of vertices have color $\pi$ for some permutation $\pi \in S_3$.
If $\pi(1)=1$, $\pi(2)=2$, $\pi(3)=3$, then we let $b_{i_1}$ and $b_{i_2}$ be the two vertices of $H_1$ with the smallest labels and with $i_1<i_2$.
Considering the triples $(b_{i_1},b_{i_2},b_{i_3})$ for every $b_{i_3}$ from $H_1$ with $i_3 > i_2$, we see that, by the choice of the coloring, all such points $b_{i_3}$ lie on the line $L=\overline{b_{i_1}w(b_{i_1}b_{i_2})}$. 
We take all these $2R-2 \geq R$ black points and let $H_2$ be the 1-subdivision of $K_R$ containing $R$ of these black points that is a subgraph of $H_1$.
By the choice of $G$, the line $L$ does not contain any white point of $H_2$, since such a point would be connected by two consecutive edges of $G$ on $L$.
We proceed analogously for other permutations $\pi \in S_3$, always obtaining the 1-subdivision of $K_R$ with the desired properties.

We now color each 4-tuple $(b_{i_1},b_{i_2},b_{i_3},b_{i_4})$ of black vertices from $H_2$ with $1 \leq i_1 < i_2 < i_3 < i_4 \leq R$ by one of $25=1+4!$ colors as follows.
We color $(b_{i_1},b_{i_2},b_{i_3},b_{i_4})$ by a color $\pi$ for a permutation $\pi \in S_4$ if $w(b_{i_{\pi(1)}},b_{i_{\pi(2)}})$ and $w(b_{i_{\pi(3)}},b_{i_{\pi(4)}})$ lie on a line from $\mathcal{L}$ containing $b_{i_{\pi(1)}}$.
If there are more choices, we select the color arbitrarily.
If there is no such permutation $\pi$, we color $(b_{i_1},b_{i_2},b_{i_3},b_{i_4})$ with color $\emptyset$.
By the choice of $R$, there is a $2k$-tuple of black vertices from $H_2$ with all 4-tuples monochromatic.
We let $H_3$ be the 1-subdivision of $K_{2k}$ that is a subgraph of~$H_2$ and contains these $2k$ black vertices.

We show that there is a 1-subdivision $H_4$ of $K_k$ that is a subgraph of $H_3$ such that no line containing an edge $H_4$ contains another vertex of $H_4$.
Let $b_i$ and $w(b_i,b_j)$ be the two vertices of $H_3$ forming an edge of $H_3$ contained in a line $L'$ from $\mathcal{L}$.
Then, $L'$ does not contain another black vertex from $H_3$, since $H_3$ is a subgraph of $H_2$.

Assume all 4-tuples of black vertices from $H_3$ are colored with $\emptyset$.
If $L'$ contains some other white vertex, then, by the choice of $G$, it is $w(b_{i'},b_{j'})$ for some $b_{i'},b_{j'}$ with $i',j' \notin \{i,j\}$ and there is a 4-tuple containing $b_i,b_j,b_{i'},b_{j'}$ that is not colored with $\emptyset$, which is impossible. 
We can then set $H_4$ as any 1-subdivision of $K_k$ that is a subgraph of $H_3$.

Thus, we assume that all 4-tuples of black vertices from $H_3$ are colored with $\pi$ for some permutation $\pi \in S_4$. 
We show that all these cases are impossible.
Let $b'_1,\dots,b'_{2k}$ be the vertices of $H_3$ such that the label of $b'_i$ in $H$ is smaller than the label of $b'_j$ in $H$ if and only if $i<j$.
First, we observe that, since $2k \geq 6$, either $\pi(1)$ and $\pi(2)$ are consecutive in $<$ or $\pi(3)$ and $\pi(4)$ are consecutive in $<$.
If, for example, $\pi(1)<\pi(3)<\pi(2)<\pi(4)$, then by considering 4-tuples $(b'_1,b'_3,b'_5,b'_6)$ and $(b'_1,b'_4,b'_5,b'_6)$, we see that the line containing the edge $b'_1w(b'_1,b'_5)$ also contains white points $w(b'_3,b'_6)$ and $w(b'_4,b'_6)$.
On the other hand, by considering the 4-tuples $(b'_2,b'_3,b'_5,b'_6)$ and $(b'_2,b'_4,b'_5,b'_6)$, the line containing the edge $b'_2w(b'_2,b'_5)$ also contains white points $w(b'_3,b'_6)$ and $w(b'_4,b'_6)$.
This is impossible as we then have a line containing an edge of $H_3$ and two black points of $H_3$.
The other cases are analogous.
In fact, $\pi(3)$ and $\pi(4)$ have to be consecutive in $<$.
If, for example, $\pi(3)<\pi(1)<\pi(2)<\pi(4)$, then the 4-tuples $(b'_1,b'_2,b'_4,b'_5)$, $(b'_1,b'_2,b'_4,b'_6)$, $(b'_1,b'_3,b'_4,b'_5)$, and $(b'_1,b'_3,b'_4,b'_6)$ similarly as before  imply that there is a line containing $b'_2w(b'_2b'_4)$ and $b'_3w(b'_3b'_4)$, which is forbidden in $H_3$. 
The remaining cases are again analogous.
Similarly, if $\pi(1)$ and $\pi(2)$ are consecutive, but $\pi(3)$ and $\pi(4)$ are not, then we have, say, $\pi(3)<\pi(1)<\pi(2)<\pi(4)$.
If we consider the 4-tuples $(b'_1,b'_3,b'_5,b'_6)$, $(b'_2,b'_3,b'_5,b'_6)$, $(b'_1,b'_4,b'_5,b'_6)$, and $(b'_2,b'_4,b'_5,b'_6)$, then we obtain a line that contains $b'_3w(b'_3b'_5)$ and $b'_4w(b'_4b'_5)$, which is impossible.
Finally, we show that the last case when both $\pi(1)$ and $\pi(2)$ are consecutive in $<$ as well as $\pi(3)$ and $\pi(4)$ is impossible as well.
For example, if $\pi(1)<\pi(2)<\pi(3)<\pi(4)$, then it suffices to consider the 4-tuples $(b'_1,b'_2,b'_4,b'_5)$, $(b'_1,b'_2,b'_4,b'_6)$, $(b'_2,b'_3,b'_4,b'_5)$, and $(b'_2,b'_3,b'_4,b'_6)$ and find a line containing $b'_1w(b'_1b'_2)$ and $b'_2w(b'_2b'_3)$, which is forbidden in $H_3$. 

Thus, we indeed get a 1-subdivision $H_4$ of $K_k$ in $G$ such that no line containing an edge of $H_4$ contains another vertex of $H_4$. 
It follows that the vertex set of $H_4$ together with the set of lines that contain an edge of $H_4$ forms a subdivided $k$-clique in $(P,\mathcal{L})$.
It follows that if $(P,\mathcal{L})$ does not contain subdivided $k$-clique as a subconfiguration, then $G$ does not contain 1-subdivision of $K_{R'}$ as a subgraph.
In particular, $G$ contains at most $ex(m,H)$ edges.

Since $H$ is a 1-subdivision of $K_{R'}$,  Theorem~\ref{thm-TuranMaxDegree} implies that $ex(m,H) \in O(m^{3/2-1/(4R'-6)})$.
Recall that $G$ contains at least $\Omega(m\log^2{m})$ edges.
Therefore, using Theorem~\ref{thm:crossing_general} for $\alpha = 1/2 - 1/(4R'-6)$, we get that 
\[\textrm{cr}(G)\geq c'\frac{(I-n)^{2+\frac{1}{1/2-1/(4R'-6)}}}{m^{1+\frac{1}{1/2-1/(4R'-6)}}} = c'\frac{(I-n)^{\frac{4R'-7}{R'-2}}}{m^{\frac{3R'-5}{R'-2}}} \geq c'\frac{(I/4)^{\frac{4R'-7}{R'-2}}}{m^{\frac{3R'-5}{R'-2}}}\]
for some constant $c'=c'(k)>0$.
On the other hand, since any two lines intersect at most once, we also have $\textrm{cr}(G)\leq \binom{n}{2}$.
Combining these upper and lower bounds on $\textrm{cr}(G)$ gives
\[I \in O\left(n^{\frac{2R'-4}{4R'-7}}m^{\frac{3R'-5}{4R'-7}}\right) = O\left(n^{1/2-\frac{1}{8R'-14}}m^{3/4-\frac{1}{16R'-28}}\right),\]
which implies the result.
\end{proof}

\section{Proof of Theorem~\ref{thm-lowerbound}}

We show that there exists a point-line configuration $(P,\mathcal{L})$ such that $|P|=|\mathcal{L}|=n$, $|I(P,\mathcal{L})|\geq n^{\frac{5}{4}-\frac{1}{2k}+o(1)}$, and the incidence graph of $(P,\mathcal{L})$ does not contain a subconfiguration isomorphic to a subdivided $k$-clique.

Following the proof of Lemma~9 from~\cite{tomSuk21}, we use a random subset of the \emph{standard point-line configuration} $(P_0, \mathcal{L}_0)$ with $\Theta(N^{4/3})$ incidences for some $N$ to be chosen later where
\[P_0 = \{(a,b) \in \mathbb{N}^2 \colon a < N^{1/3}, b<N^{2/3}\}\]
and
\[\mathcal{L}_0 = \{\{(x,y) \in \mathbb{R}^2 \colon y=ax+b\} \colon a < N^{1/3}, b<N^{2/3}\}.\]

Let $B_0$ be the incidence graph of $(P_0,\mathcal{L}_0)$ and $H$ be the graph on $P_0$ in which $p$ and $p'$ are joined by an edge if $p,p'\in L$ for some $L\in \mathcal{L}_0$.

\begin{lemma}[Claim 10 in \cite{tomSuk21}]
\label{lem-common}
Let $p,p'\in P_0$ be distinct vertices. The number of common neighbors of $p$ and $p'$ in $H$ is at most $N^{1/3+o(1)}$.
\end{lemma}

The next proposition is analogous to Claim 11 in \cite{tomSuk21}, but for subdivided cliques instead of cycles.

\begin{lemma}
\label{lem-copies}
For every integer $k\geq 3$, the number of incidence graphs of subdivided $k$-cliques in $B_0$ is at most $N^{(k^2+5k)/6+o(1)}$.
\end{lemma}

\begin{proof}
We first pick the $k$ black vertices for which we have at most $N^k$ choices.
Then, we pick the $\binom{k}{2} = \frac{k^2-k}{2}$ white vertices.
Since each such white vertex lies in a common neighborhood of two black vertices in $H$, Lemma \ref{lem-common} implies that we have at most $N^{1/3+o(1)}$ choices for each white vertex.
Altogether, we have at most $N^{k+(k^2-k)/6+o(1)} = N^{(k^2+5k)/6+o(1)}$ incidence graphs of subdivided $k$-cliques in $B_0$.
\end{proof}

\begin{proof}[Proof of Theorem~\ref{thm-lowerbound}]
Let $S$ denote the number of subdivided $k$-cliques in $B_0$. 
Let $q\in (0,1)$, and let $P'$ and $\mathcal{L'}$ be subsets of $P_0$ and $\mathcal{L}_0$, respectively, in which each element is present independently with probability $q\geq n^{-1/3+\varepsilon}$ for some $\varepsilon > 0$.
We use $B'$ to denote the subgraph of $B_0$ induced by $P'\cup \mathcal{L}'$. 
Further, let $X$ be the number of incidence graphs of subdivided $k$-cliques in $B'$. 
Then, $\mathbb{E}(|P'|)=\mathbb{E}(|\mathcal{L}'|)=qN$ and $\mathbb{E}(|X|)=q^{(3k^2-k)/2}S$, since $k+\binom{k}{2}+2\binom{k}{2}$ points and lines participate in total in each subdivided $k$-clique.

Let $\mathcal{A}$ be the event that $B'$ satisfies the following properties:
\begin{itemize}
\item $\frac{qN}{2} < |P'|,|\mathcal{L}'| < 2qN$,
\item the degree of every vertex of $B'$ is at most $2qN^{1/3}$,
\item there are at least $\frac{qN}{4}$ points and lines in $B'$ whose degree in $B'$ is between $\frac{qN^{1/3}}{4}$ and $2qN^{1/3}$.
\end{itemize}

Then, by Chernoff's inequality, we have $\Pr(\mathcal{A})>2/3$. 
If $\mathcal{A}$ holds, then we have $|E(B')|\geq \frac{q^2N^{4/3}}{16}$. 
Choose $q$ such that $\mathbb{E}(X)\leq \frac{qN}{128}$. 
We have $\mathbb{E}(|X|)=q^{(3k^2-k)/2}S$ and $S\leq N^{(k^2+5k)/6+o(1)}$ by Lemma~\ref{lem-copies}, so we can choose $q=N^{-\frac{k+6}{9k+6}+o(1)}$.

By Markov's inequality, we have $X\leq \frac{qN}{64}$ with probability at least $\frac{1}{2}$. 
Then, there exists $B'$ such that $\mathcal{A}$ holds and the number of copies of subdivided $k$-cliques in $B'$ is at most $\frac{qN}{64}$. 
Delete a vertex from each copy and let $B''$ be the resulting graph with parts $P''$ and $\mathcal{L}''$. 
In this way, we delete at most $\frac{qN}{64}$ vertices and at most $\frac{qN}{64}\cdot 2qN^{1/3}$ edges.
Then, $B''$ has at least $\frac{q^2N^{4/3}}{16} - \frac{q^2N^{4/3}}{32} \geq\frac{q^2 N^{4/3}}{32}$ edges and no copy of the incidence graph of a subdivided $k$-clique.

If we choose $n$ such hat $\frac{qN}{8}\leq n\leq \frac{qN}{4}$, then $n=N^{1-\frac{k+6}{9k+6}+o(1)}$. 
By sampling random $n$-element subsets of $P''$ and $\mathcal{L}''$, we get that there exists an induced subgraph $B$ of $B''$ with parts $P$ and $\mathcal{L}$, each of size $n$, such that $B$ has at least 
\[\frac{q^2N^{4/3}}{128}=N^{\frac{4}{3}-2\frac{k+6}{9k+6}+o(1)}=n^{\frac{5}{4}-\frac{1}{2k}+o(1)} \]
edges and with no copy of the incidence graph of a subdivided $k$-clique.
\end{proof}

\bibliography{bibliography}
\bibliographystyle{plainnat}

\end{document}